\documentclass{cmslatex}
\usepackage[paperwidth=7in, paperheight=10in, margin=.875in]{geometry}
 \usepackage[backref,colorlinks,linkcolor=red,anchorcolor=green,citecolor=blue]{hyperref}
\usepackage{amsfonts,amssymb}
\usepackage{amsmath}
\usepackage{graphicx}
\usepackage{cite}
\usepackage{enumerate}
\usepackage{extpfeil}
\renewcommand{\theequation}{\arabic{section}.\arabic{equation}}

\newtheorem{question}{Question}[section]

   \allowdisplaybreaks
\begin{document}
 \title{Asymptotic formation and orbital stability of phase-locked states in Kuramoto--Lohe type synchronization models on Lie groups}


          \author{Seung-Yeon Ryoo\thanks{Mathematics Department, California Institute of Technology, Pasadena California 91125, United States, \href{mailto:sryoo@caltech.edu}{(sryoo@caltech.edu)}.}}

         \pagestyle{myheadings} \markboth{Formation and stability of phase-locked states in Kuramoto--Lohe models}{Seung-Yeon Ryoo} \maketitle

          \begin{abstract}
Some mathematical models of synchronization, such as the Kuramoto model (1975) and its generalizations pioneered by Lohe (2009), are formulated as ordinary differential equations describing populations of particles on Lie groups with locally attractive interactions. We suggest a model of synchronization on Lie groups and present a framework to understand the formation of phase-locked states and their orbital stability. This is a sequel to a previous joint work with Ha and Ko (2017).
          \end{abstract}

          \section{Introduction}
\setcounter{equation}{0}

Let $G$ be a Lie group with Lie algebra $(\mathfrak{g},[,])$. We will define the \emph{generalized Kuramoto--Lohe model on} $G$ as the Cauchy problem
\begin{equation}\label{genLohe}
\begin{cases}
\displaystyle
\dot{X}_i=(dR_{X_i})_e\left(H_i+\frac{\kappa}{N}\sum_{j=1}^N \phi(X_j X_i^{-1})\right),\\
X_i(0)=X_i^0,
\end{cases}
\quad i=1,\cdots,N,
\end{equation}
for the solution $\{X_1,\cdots, X_N\}$, where $X_1^0,\cdots,X_N^0\in G$ are the initial data, $H_1,\cdots,H_N\in \mathfrak{g}$ are the intrinsic Hamiltonians, and $\kappa>0$ is the coupling strength. Also, $e\in G$ is the identity element of $G$, and for $g\in G$, $R_g:G\rightarrow G$ denotes right multiplication $R_g(x):=xg$ (and similarly $L_g:G\rightarrow G$ denotes left multiplication $L_g(x):=gx$). Finally, $\phi:G\to\mathfrak{g}$ is a map describing the pairwise interaction of the oscillators and satisfies
\begin{equation}\tag{H}\label{Hypo}
    \phi \mathrm{~is~} C^1, ~ \phi(e)=0, ~ \mathrm{all~eigenvalues~of~}(d\phi)_e:\mathfrak{g}\to\mathfrak{g} \mathrm{~have~positive~real~parts}.
\end{equation}

As will be explained in subsection \ref{subsec:motivation} below, this model generalizes the well-known Kuramoto model \cite{kuramoto1975international,kuramoto2003chemical} and its non-abelian generalization to the unitary groups by Lohe \cite{lohe2009non,lohe2010quantum}.

The standard Cauchy--Lipschitz theory guarantees the existence of a local $C^1$-solution to \eqref{genLohe}. However, global existence of a solution may fail if the initial data are far apart: see \cite[Remark 2.3]{ha2017emergent} for an example.

Note that, for $g\in G$ and $J$ in the center of $\mathfrak{g}$, the system \eqref{genLohe} admits the Galilean transformation
\[
X_i^0\mapsto X_i^0g,\quad X_i(t)\mapsto X_i(t)ge^{Jt},\quad H_i\mapsto H_i+J.
\]
Therefore, it is natural to consider right-invariant metrics on $G$ for the purposes of analyzing \eqref{genLohe}, and to consider equilibria of \eqref{genLohe} as the solutions for which the ratios $X_iX_j^{-1}$ are constant. More precisely, endow $\mathfrak{g}$ with a Euclidean norm $|\cdot|$, extend this to a right-invariant Riemannian metric on $G$, and let $d:G\times G\to [0,\infty)$ be the corresponding right-invariant geodesic distance. Also, we will call equilibria ``phase-locked states", as defined below.

\begin{definition}
A solution $\{X_i\}_{i=1}^N$ to \eqref{genLohe} is a \emph{phase-locked state} if $X_i(t)X_j(t)^{-1}$ is constant in $t$ for each $i,j=1,\cdots,N$.
\end{definition}

The following characterization is essentially due to Lohe \cite{lohe2009non}.
\begin{proposition}[Characterization of phase-locked states, {\cite[p.6]{lohe2009non}}]\label{Prop:pls}
A solution $\{X_i\}_{i=1}^N$ to \eqref{genLohe} is a phase-locked state if and only if $X_i(t)=X_i^\infty \exp(\Lambda t)$, where $\exp$ is the Lie group exponential and $X_i^\infty\in G$, $i=1,\cdots,N$, and $\Lambda\in \mathfrak{g}$ satisfy
\begin{equation}\label{pls}
\operatorname{Ad}_{X_i^\infty}\Lambda = H_i+\frac{\kappa}{N}\sum_{j=1}^N\phi(X_j^\infty (X_i^\infty)^{-1}),\quad i=1,\cdots,N,
\end{equation}
where $\operatorname{Ad}_g= (dL_g)_{g^{-1}}\circ (dR_{g^{-1}})_e$ is the adjoint.
\end{proposition}

The main theorems of this paper are as follows. The first main theorem says that if the intrinsic Hamiltonians are identical, then a unique global solution exists and exhibits complete synchronization.

\begin{theorem}\label{idmainthm}
There exist constants $C,c,c_1,c_2>0$ depending on $G$, $\mathfrak{g}$, and $\phi$ such that the following is true.
\begin{enumerate}[(1)]
    \item Let the coupling strength $\kappa$ and intrinsic Hamiltonians $H_i$ satisfy
\[
\kappa>c|H|,\quad H_i=H,~i=1,\cdots,N,
\]
for some fixed $H\in \mathfrak{g}$, and let the initial condition $\{X_i^0\}_{i=1}^N$ satisfy
\[
d(X_i^0,X_j^0)<c_1,\quad i,j=1,\cdots,N.
\]
Then the following statements hold.
\begin{enumerate}
    \item (Global existence of solutions) There is a global solution $\{X_i(t)\}_{i=1}^N$ to \eqref{genLohe}.
    \item (Exponential synchronization) We have
    \[
    d(X_i(t),X_j(t))\le Ce^{-c_2\kappa t},\quad i,j=1,\cdots,N.
    \]
\end{enumerate}
    \item Suppose that the interaction function $\phi$ satisfies in addition
    \[
    \operatorname{Ad}_g\phi(h)=\phi(ghg^{-1}),\quad g,h\in G.
    \]
    Let the coupling strength $\kappa$ and intrinsic Hamiltonians $H_i$ satisfy
\[
\kappa>0,\quad H_i=H,~i=1,\cdots,N,
\]
for some fixed $H\in \mathfrak{g}$, and let the initial condition $\{X_i^0\}_{i=1}^N$ satisfy
\[
d(X_i^0,X_j^0)<c_1.
\]
Then the following statements hold.
\begin{enumerate}
    \item (Global existence of solutions) There is a global solution $\{X_i(t)\}_{i=1}^N$ to \eqref{genLohe}.
    \item (Solution operator splitting) $\tilde X_i(t)=\exp(-Ht)X_i(t)$ is a solution to \eqref{genLohe} with $H_i=0$, $\tilde X_i(0)=X_i^0$.
    \item (Exponential synchronization) We have
    \[
    d(\exp(-Ht)X_i(t),\exp(-Ht)X_j(t))\le Ce^{-c_2\kappa t},\quad i,j=1,\cdots,N.
    \]
\end{enumerate}
\end{enumerate}

\end{theorem}

The second main theorem says that if the coupling strength $\kappa$ is sufficiently large compared to the magnitude of the Hamiltonians $|H_i|$, then phase-locked states, with each particle being sufficiently close together (measured by $d$), exist and are unique up to right-multiplication. Also, if the initial data are sufficiently close together, then a global solution exists and exponentially converges towards a right-multiplication of the aforementioned phase-locked state.

\begin{theorem}\label{orbitalstabilitymainthm}
There exist constants $C$ and $c_2$ depending on $G$, $\mathfrak{g}$, and $\phi$, and a sufficiently small constant $c_1>0$ and a sufficiently large constant $c>0$ depending on $G$, $\mathfrak{g}$, and $\phi$ such that\footnote{This wording means that in the proof we will first select $C$ and $c_2$, then select $c_1$ depending on $C$ and $c_2$, then select $c$ depending on $C$, $c_2$, and $c_1$.} if
\begin{equation}\label{kappa}
    \kappa>c \|H\|_\infty,\quad \|H\|_\infty\coloneqq \max_{i=1,\cdots,N}|H_i|,
\end{equation}
and if $\{X_i(t)\}_{i=1}^N$ is a solution to \eqref{genLohe} with initial data $\{X_i^0\}_{i=1}^N$ satisfying
\[
    d(X_i^0,X_j^0)<c_1,\quad\forall i,j=1,\cdots,N,
\]
then the following holds.
\begin{enumerate}[(1)]
    \item (Existence and uniqueness of a phase-locked state) There exists a phase-locked state $\{X_i^\infty \exp(\Lambda t)\}_{i=1}^N$, with $X_i^\infty\in G$, $i=1,\cdots,N$, $\Lambda\in \mathfrak{g}$ satisfying \eqref{pls}, such that
    \[
    d(X_i^\infty,X_j^\infty)\le \frac{C\|H\|_\infty}{\kappa}.
    \]
    This is unique up to right-multiplication in a larger set, i.e., if $\{\tilde{X}_i(t)\}_{i=1}^N$ is a phase-locked state with
    \[
    d(\tilde{X}_i^0,\tilde{X}_j^0)<c_1,\quad\forall i,j=1,\cdots,N,
    \]
    then $\tilde{X}_i(t)=X_i^\infty \exp(\Lambda t)g$, $i=1,\cdots,N$, for some fixed $g\in G$. (This larger set will be the basin of attraction for this phase-locked state.)
    \item (Global existence and local stability) There exists a unique global solution $X(t)=\{X_i(t)\}_{i=1}^N$ to \eqref{genLohe}, with
    \[
    d(X_i(t),X_j(t))<c_1,~t\ge 0,\quad \forall i,j=1,\cdots,N,
    \]
    and
    \[
    \limsup_{t\to\infty}d(X_i(t),X_j(t))\le \frac{C\|H\|_\infty}{\kappa},~t\ge 0,\quad \forall i,j=1,\cdots,N.
    \]
    \item (Asymptotic phase-locking) We have
    \[
        \lim_{t\to\infty}X_i(t)X_j(t)^{-1}=X_i^\infty (X_j^{\infty})^{-1},
    \]
    the convergence rate being $O_{G}(e^{-\kappa c_2t})$.\footnote{We will use the following (standard) asymptotic notation. For $P, Q>0$, $P=O(Q)$ and $Q=\Omega(P)$ mean that $P\le KQ$ for a universal constant $K\in (0,\infty)$. If we need to allow for dependence on parameters, we indicate this by subscripts. For example, in the presence of auxiliary parameters $\psi, \xi$, the notations $P=O_{\psi,\xi}(Q)$ and $Q=\Omega_{\psi,\xi}(P)$ mean that $P\le K(\psi,\xi)Q$ where $K(\psi,\xi)\in (0,\infty)$ may depend only on $\psi$ and $\xi$.}
    \item (Orbital stability) The solution $X(t)=\{X_i(t)\}_{i=1}^N$ exponentially converges to a right-multiplication of the phase-locked state of (1), i.e., there is a phase-locked state $\{Z_i^\infty\exp(M t)\}_{i=1}^N$ as given in Proposition \ref{Prop:pls} such that
    \[
        d(X_i(t),Z_i^\infty \exp(M t))=O_{G,\phi}(e^{-\kappa c_2 t}),\quad t\ge 0.
    \]
    \item (Exponential synchronization of the normalized speeds) We have
    \[
        (dL_{X_i(t)}^{-1})_e\dot{X}_i(t)-M=O_{G,\phi}(\|H\|e^{-\kappa c_2t}),\quad i=1,\cdots,N.
    \]
    
\end{enumerate}
\end{theorem}

In the rest of the introduction, we will discuss the motivation and significance of the model \eqref{genLohe} and Theorem \ref{orbitalstabilitymainthm}, and we will provide possible lines of further research.

\subsection{Motivation and significance of the generalized Kuramoto--Lohe model and Theorem \ref{orbitalstabilitymainthm}.}\label{subsec:motivation}
The \emph{Kuramoto model}  \cite{kuramoto1975international,kuramoto2003chemical} is the Cauchy problem for the real variables $\{\theta_i(t)\}_{i=1}^N$ given by the nonlinear system of ordinary differential equations
\begin{equation}\label{Ku}
\begin{cases}
\dot{\theta}_i(t)=\omega_i+\frac{\kappa}{N}\sum_{j=1}^N \sin(\theta_j(t)-\theta_i(t))\\
\theta_i(0)=\theta_i^0
\end{cases}
,\quad i=1,\cdots,N,
\end{equation}
where $\omega_1,\cdots,\omega_N\in \mathbb{R}$ are the intrinsic frequencies, $\kappa>0$ is the coupling strength, and $\theta_1^0,\cdots,\theta_N^0\in \mathbb{R}$ are the initial data.

We may consider the Kuramoto model to be a special case of the generalized Kuramoto--Lohe model, with Lie group $G=\mathbb{S}^1$ and interaction function $\phi(\theta)=\sin\theta$ (we are considering the dynamics on the universal cover $\mathbb{R}$ of $\mathbb{S}^1$). Since the interactions are pairwise attractive, we may expect that if the coupling strength $\kappa$ is large enough compared to the differences $\omega_i-\omega_j$ in the intrinsic frequencies, then the population should become more and more ordered as time goes by.

Theorem \ref{orbitalstabilitymainthm} was proven for the Kuramoto model by Ha, Ha, and Kim \cite{ha2010complete} and Choi, Ha, Jung, and Kim \cite{choi2012asymptotic}, with $c_1=\frac{\pi}{4}$. Namely, if the particles all lie on the same quarter-circle, then they stay within a quarter-circle for all time, their relative phases $\theta_i-\theta_j$ converge to a unique value, $\theta_i-\theta_i^\infty-\nu t$ converges to zero for some speed $\nu$ and some phase-locked state $\{\theta_i^\infty\}$ unique up to an additive constant, and $\dot{\theta}_i-\dot{\theta}_j$ converges exponentially to zero (one can extend the quarter-circle to a half-circle with some complications).

Lohe \cite{lohe2009non,lohe2010quantum} suggested a nonabelian generalization of the Kuramoto model to the unitary groups $U(d)$  as follows: the \emph{Lohe model} describes the dynamics of $\{U_i(t)\}_{i=1}^N\subset U(d)$ as
\begin{equation}\label{UnitaryLohe}
\dot{U}_i U_i^\dagger=iH_i+\frac{\kappa}{2N}\sum_{j=1}^N\left(U_jU_i^\dagger-U_iU_j^\dagger\right),
\end{equation}
where $\kappa>0$ is again the coupling strength, the Hermitian matrix $H_i$ is the Hamiltonian of $U_i$, and $\dagger$ denotes the Hermitian conjugate. It can be checked that the Kuramoto model is the $U(1)$ case of the Lohe model \eqref{UnitaryLohe}, and the Lohe model \eqref{UnitaryLohe} is in turn the case $G=U(d)$, $\phi(U)=\frac 12 (U-U^\dagger)$ of the generalized Kuramoto--Lohe model. Lohe \cite{lohe2009non} numerically observed Theorem \ref{orbitalstabilitymainthm} for the $U(2)$ model, and Ha and the author \cite{ha2016emergence} gave a rigorous proof of Theorem \ref{orbitalstabilitymainthm} for all unitary groups $U(d)$.

It was suggested by Lohe \cite{lohe2009non} that the model \eqref{UnitaryLohe} can be extended to other matrix groups. Here is a direct quotation from \cite[p.5]{lohe2009non}:
\begin{quote}
{\em \ldots There also exist non-Abelian Kuramoto models for the symplectic groups, with defining equations similar to [\eqref{UnitaryLohe}], and hence for all classical compact Lie groups. By choosing an indefinite metric one may extend these models to the noncompact case, for which trajectories are unbounded, however, we do not discuss synchronization or other properties of these particular models here\ldots}
\end{quote}
In \cite{ha2017emergent}, Ha, Ko, and the author interpreted this as a model on $GL_n(\mathbb{C})$  as follows:
\begin{equation}\label{MatrixLohe}
\dot{X}_i X_i^{-1}=H_i+\frac{\kappa}{2N}\sum_{j=1}^N\left(X_jX_i^{-1}-X_iX_j^{-1}\right),
\end{equation}
and characterized various matrix Lie groups, such as $U(d)$, $O(m,n)$, $SO(m,n)$, $SP(d)$, $SL_2(\mathbb{R})$, on which this model defines a closed trajectory. We were also able to establish Theorem \ref{orbitalstabilitymainthm} for the model \eqref{MatrixLohe} in \cite{ha2017emergent}.

We remark that the following variant model on $GL_n(\mathbb{C})$
\begin{equation}\label{DeVille}
\dot{X}_i X_i^{-1}=H_i+\frac{1}{2}\sum_{j=1}^N\gamma_{ij}\left(f(X_jX_i^{-1})-f(X_iX_j^{-1})\right),
\end{equation}
was considered by DeVille \cite{deville2019synchronization}.

The author was thus led to propose the model \eqref{genLohe} in the general setting of Lie groups $G$ with locally attractive interaction $\phi$, and to consider whether Theorem \ref{orbitalstabilitymainthm} holds true in this setting. Part (2) of Theorem \ref{orbitalstabilitymainthm} is unsurprising: as the particles are pairwise attractive, if they are initially close together, then they should stay close together for all times. Thus, part (4) of Theorem \ref{orbitalstabilitymainthm}, the orbital stability statement, is the crux of Theorem \ref{orbitalstabilitymainthm}: it only requires \eqref{Hypo}, namely that $\phi$ is $C^1$ and that it is locally attractive in the sense that $(d\phi)_e$ has only eigenvalues with positive real parts. Note that higher order considerations, such as $C^2$ regularity of $\phi$ or curvature properties of $G$, are irrelevant.

Thus, as long as we define any model in the form of \eqref{genLohe}, by choosing a Lie group $G$ and an interaction function $\phi:G\to\mathfrak{g}$ to satisfy \eqref{Hypo}, then this model will automatically have the stability properties of Theorem \ref{orbitalstabilitymainthm}. This grants us significant liberty in defining locally well-behaved synchronization models: we only need to choose a Lie group $G$ and an interaction function $\phi:G\to\mathfrak{g}$ satisfying \eqref{Hypo}.

For example, it was suggested by an anonymous referee for the previous paper \cite{ha2016emergence} that we consider the model on $G=SU(d), SL_d(\mathbb{C}), SL_d(\mathbb{R})$ given by
\begin{equation}\label{SLnLohe}
\dot{X}_i X_i^{-1}=H_i+\frac{\kappa}{2N}\sum_{j=1}^N\left(X_jX_i^{-1}-X_iX_j^{-1}-\frac 1d \mathrm{Tr}[X_jX_i^{-1}]I_d+\frac 1d \mathrm{Tr}[X_iX_j^{-1}]I_d\right).
\end{equation}
One can check that \eqref{Hypo} holds for this interaction function, so that Theorem \ref{orbitalstabilitymainthm} holds for \eqref{SLnLohe}.

Here is an intuitive heuristic as to why Theorem \ref{orbitalstabilitymainthm} should be true. If the particles are concentrated locally, say around the identity $e\in G$, then we may linearize the system by setting
\[
X_i=\exp(Y_i)\quad i=1,\cdots,N,
\]
and the governing equation \eqref{genLohe} linearizes to
\begin{equation}\label{linearization}
\dot{Y}_i = H_i+\frac{\kappa}{N}\sum_{k=1}^N (d\phi)_e(Y_k-Y_i),\quad i=1,\cdots,N.
\end{equation}
This is an exactly solvable linear differential equation. In particular, the dynamics for the differences $Y_i-Y_j$ is simple:
\[
\dot{Y}_i-\dot{Y}_j = H_i-H_j -\kappa (d\phi)_e(Y_i-Y_j),\quad i,j=1,\cdots,N.
\]
As $(d\phi)_e$ has only eigenvalues with positive real parts, if we make $\kappa$ large enough compared to $|H_i-H_j|$, then $Y_i-Y_j$ should converge exponentially to a fixed value.

Thus, the proof of Theorem \ref{orbitalstabilitymainthm} for the nonlinear model \eqref{genLohe} will be a ``nonlinearization'' of the above heuristic argument, i.e., we will check that the nonlinear effects are controllable compared to the dominant linear term. One caveat of this nonlinearization is that we will need to require that $\kappa$ is large compared to the $|H_i|$'s as in \eqref{kappa}, instead of requiring $\kappa$ to be large compared to the $|H_i-H_j|$'s. There will also be the complication that the $X_i$'s may not stay in the neighborhood of a single point for all time, so we will work with the logarithms of the \emph{ratios} $X_iX_j^{-1}$, which will stay close to the identity $e\in G$.

\subsection{Directions for further research.}
One reason for considering generalizations of the Kuramoto model \eqref{Ku} should come from a search for new phenomena not observed in the Kuramoto model. Thus, since Theorem \ref{orbitalstabilitymainthm} is true not only for the Kuramoto model but also for its Lie group generalizations in the form of \eqref{genLohe}, the merit of studying such models in the form of \eqref{genLohe} will have to come from the behavior of the model when the particles are not extremely concentrated. For example, DeVille \cite{deville2019synchronization} presents several special solutions to \eqref{DeVille} and shows that it has multiple attractors under rich network topologies, while Ritchie, Lohe, and Williams \cite{ritchie2018synchronization} present a detailed study of the global solutions of the special case $G=SO(1,1)$.

In the proof of Theorem \ref{orbitalstabilitymainthm}, one proves the statements in the order of (2), (3), (1), (4), and (5), with possibly smaller $c_1$ and larger $c$ (and consequently larger $\kappa$) along the way. Thus there might be a sequence of phase transitions
\[
X_iX_j^{-1}\mbox{ bounded }\Rightarrow X_iX_j^{-1}\mbox{ converges }\Rightarrow (X_i\to X_j^\infty \exp(\Lambda t)),
\]
as we increase the coupling strength $\kappa$. It would be interesting to see whether the first phase-transition actually occurs. That is, for small $\kappa$, the solutions might stay bounded from each other with their ratios not converging, possibly exhibiting chaotic behavior. Observing such `synchronous chaos' in a prototypical model such as \eqref{UnitaryLohe} would be interesting.
\begin{question}[Synchronous chaos]\label{chaos}
Is it possible that in a semi-local regime with intermediate coupling strength $\kappa$:
\[
c_1'<\max_{i,j=1,\cdots,N}d(X_i^0,X_j^0)<c_1,\quad c\|H\|_\infty <\kappa<c'\|H\|_\infty,
\]
with different constants $c_1'$ and $c'$, $X_iX_j^{-1}$ stays bounded yet fail to converge? Does this occur for the model \eqref{UnitaryLohe}?
\end{question}\newline
See Remark \ref{chaosremark} for another formulation of this question.

We remark that in general, as we only guarantee the local property of the interaction function $\phi$, it is in general impossible to guarantee asymptotic phase-locking for generic initial data.

\begin{example}
By \eqref{ratioODE} below, for $N=2$ and $H_1=H_2=0$, we have
\begin{align}
\begin{aligned}
\frac{d}{dt}\left(X_1 X_2^{-1}\right)
=&\frac{\kappa}{2}\left((dR_{X_1X_2^{-1}})_e \phi(X_2 X_1^{-1})-(dL_{X_1X_2^{-1}})_e\frac{\kappa}{2} \phi(X_1 X_2^{-1})\right).
\end{aligned}
\end{align}
In particular, if $\phi(g^{-1})=-\phi(g)$ and $G$ is abelian, we have
\begin{align}
\begin{aligned}
\frac{d}{dt}\left(X_1 X_2^{-1}\right)
=&-\kappa\left((dR_{X_1X_2^{-1}})_e \phi(X_1 X_2^{-1})\right).
\end{aligned}
\end{align}
For the $2$-torus $G=\mathbb{T}^2=\mathbb{S}^1\times \mathbb{S}^1$, it is not hard to construct $\phi:\mathbb{T}^2\to \mathbb{R}^2$ with the following properties:
\begin{itemize}
\item $\phi$ satisfies \eqref{Hypo} near the identity $(1,1)$ of $\mathbb{T}^2$ (we are viewing $\mathbb{S}^1\subset \mathbb{C}$ and $\mathbb{T}^2\subset \mathbb{C}^2$),
\item $\phi(e^{\mathrm{i}\theta},e^{\mathrm{i}\varphi})=(0,\mathrm{i})$, for $\theta\in \left(\frac\pi 2-\varepsilon,\frac\pi 2+\varepsilon\right)$ and $\varphi\in \mathbb{R}$, where $\varepsilon\in \left(0,\frac\pi 2\right)$, and
\item $\phi(g^{-1})=-\phi(g)$, $g\in \mathbb{T}^2$.
\end{itemize}
Then, the vectorfield $g\mapsto (dR_g)_e\phi(g)$ induces a closed circular vectorfield on $\{e^{\mathrm{i}\theta}\}\times \mathbb{S}^1$ for each $\theta\in \left(\frac\pi 2-\varepsilon,\frac\pi 2+\varepsilon\right)$, so that if $X_1^0=(e^{\mathrm{i}\theta},1)$ and $X_2^0=(1,1)$, then $X_1(t)X_2(t)^{-1}=(e^{\mathrm{i}\theta},e^{-\mathrm{i}\kappa t})$. No matter how high $\kappa$ is, there is a nonempty open set of initial data for which asymptotic phase-locking never happens. Hence asymptotic phase-locking for generic initial data in the large coupling regime is impossible in this case.
\end{example}

We remark that some models which are qualitatively different from the Kuramoto model can still be written in the form \eqref{genLohe} with some modifications. For example, the Cucker--Smale model \cite{cucker2007emergent,cucker2007mathematics}
\begin{equation}\label{CS}
    \begin{cases}
    \dot{x}_i = v_i,\\
    \dot{v}_i = \frac{\kappa}{N}\sum_{j=1}^N \psi(|x_i-x_j|)(v_j-v_i),
    \end{cases}
    t>0,~i=1,\cdots,N,
\end{equation}
where $\psi:\mathbb{R}^+\to\mathbb{R}$ is a function, is, up to a uniform translation in the $x$-variable, in the form of \eqref{genLohe} with $(x,v)\in G=T\mathbb{R}^d=\mathbb{R}^d\times \mathbb{R}^d$, with the interaction function $\phi$ now failing to have only eigenvalues with positive real parts, and fails to be globally defined if $\psi$ possesses a singularity at $0$. Another example is the Swarmalator model proposed by O'Keeffe, Hong, and Strogatz \cite{o2017oscillators}:
\begin{equation}\label{Swarm}
    \begin{cases}
    \dot{x}_i = v_i+\frac 1N \sum_{j=1}^N [ I_{att}(x_j-x_i) F(\theta_j-\theta_i)-I_{rep}(x_j-x_i)],\\
    \dot{\theta}_i = \omega_i+\frac{\kappa}{N}\sum_{j=1}^N H_{att}(\theta_j-\theta_i)G(x_j-x_i),
    \end{cases}
    t>0,~i=1,\cdots,N.
\end{equation}
This is again in the form of \eqref{genLohe} with $(x,\theta)\in G = \mathbb{R}^d\times \mathbb{S}^1$. Thus it is natural to pose the following question.

\begin{question} What is the behavior of the model \eqref{genLohe} for different $\phi$?
\begin{enumerate}[(1)]
    \item Let $\phi:G\to\mathfrak{g}$ be smooth but not satisfy \eqref{Hypo}. To what extent can the local behavior of \eqref{genLohe} be explained by the behavior of the linearized model \eqref{linearization}?
    \item Let $G=\mathbb{R}^d$ and let $\phi:\mathbb{R}^d\setminus\{0\}\to\mathbb{R}^d$ possess a singularity at the origin. Are there certain types of singularities of $\phi$ for which solutions near the origin exhibit nice behavior?
\end{enumerate}
\end{question}

Another drawback of Theorem \ref{orbitalstabilitymainthm} is a possible deterioration of the statements for large dimensions. In many cases, we have a family of Lie groups, such as the family of unitary groups $\{U(d)\}_{d=1}^\infty$, whose dimension grows to infinity. In this case, the size of the stability basin of Theorem \ref{orbitalstabilitymainthm} compared to the whole group diminishes exponentially as the dimension parameter grows to infinity. Thus, in practice, it would be desirable to obtain dimension-independent bounds for Theorem \ref{orbitalstabilitymainthm}. This is possible for the Lohe model \ref{UnitaryLohe}, by using the operator norm instead of the Frobenius norm, as can be checked by modifying the arguments of \cite{ha2016emergence} in a straightforward manner.
\begin{question}
For incarnations of the model \eqref{genLohe} other than \eqref{UnitaryLohe}, is it possible to obtain dimension-independent versions of Theorem \ref{orbitalstabilitymainthm}?
\end{question}

It would also be interesting to study infinite limits such as the mean-field limit, as performed for the Lohe model \eqref{UnitaryLohe} by Golse and Ha \cite{golse2019mean}, or to study stochastic variations, as performed for the Lohe matrix model \eqref{MatrixLohe} by Kim and Kim \cite{kim2020stochastic}. One could also study the effects of frustrations as in the work of Ha, Kim, Park, and the author \cite{ha2021constants}.

The rest of this paper is organized as follows. We recall basic Lie group theory and establish some notation in Section \ref{sec:Lie}, and then we prove Theorem \ref{orbitalstabilitymainthm} in Section \ref{sec:proof}.
\section{Notation and preliminaries on Lie groups and Lie algebras}\label{sec:Lie}
\setcounter{equation}{0}

We establish some notation about our Lie group $(G,\mathfrak{g})$ and recall some standard facts from the literature. 

As $(d\phi)_e:\mathfrak{g}\to\mathfrak{g}$ has only eigenvalues with positive real parts, we may assume that the right-invariant Riemannian metric on $G$, which is given as an inner product $\langle,\rangle$ on  $\mathfrak{g}$ with corresponding norm $|\cdot|=\langle\cdot,\cdot\rangle^{1/2}$, is so that for some $\lambda>0$,
\begin{equation}\label{pos}
    \langle v,(d\phi)_e(v)\rangle\ge \lambda |v|^2,\quad \forall v\in \mathfrak{g}.
\end{equation}
(This change of Riemannian metric only distorts the geodesic metric of $G$ by a constant factor, and thus does not affect the validity of Theorem \ref{orbitalstabilitymainthm}.)
We define, for $r>0$, the ball of radius $r$
\[
B_r=\{v\in \mathfrak{g}:|v|<r\}.
\]

The Lie group exponential map $\exp:\mathfrak{g}\to G$ is a local diffeomorphism at $0\in \mathfrak{g}$, with $(d\exp)_0=\operatorname{id}_\mathfrak{g}$. We may choose a radius $r>0$ depending on $G$ so that $\exp:B_r\to \exp(B_r)$ is a diffeomorphism, and consequently the logarithm $\log:\exp(B(0,r))\to B_r$ is well defined and smooth; we will use this as a coordinate chart for $G$ in this paper.

The Baker--Campbell--Hausdorff formula describes the group operation on this coordinate chart. Given $v,w\in \mathfrak{g}$, with $|v|,|w|$ sufficiently small, if we define
\begin{align}
    \begin{aligned}
    u\coloneqq& \sum_{m = 1}^\infty\frac {(-1)^{m-1}}{m}
\sum_{\begin{smallmatrix} r_1 + s_1 > 0 \\ \vdots \\ r_m + s_m > 0 \end{smallmatrix}}
\frac{[ v^{r_1} w^{s_1} v^{r_2} w^{s_2} \dotsm v^{r_m} w^{s_m} ]}{(\sum_{j = 1}^m (r_j + s_j)) \cdot \prod_{i = 1}^m r_i! s_i!}\\
    =& v+w+a(v,w),
    \end{aligned}
\end{align}
where we have used the notation
\[
[ g^{r_1} h^{s_1} \dotsm g^{r_m} h^{s_m} ] = [ \underbrace{g,[g,\dotsm[g}_{r_1} ,[ \underbrace{h,[h,\dotsm[h}_{s_1} ,\,\dotsm\, [ \underbrace{g,[g,\dotsm[g}_{r_m} ,[ \underbrace{h,[h,\dotsm h}_{s_m} ]]\dotsm]],
\]
and $a(v,w)$ denotes the sum of the terms of second or higher degree, then the power series defining $u$ is absolutely convergent within some positive radius of convergence, and satisfies
\[
\exp(u)=\exp(v)\exp(w).
\]
We may assume that the radius $r>0$ is sufficiently smaller than this radius of convergence so that for $u,v\in B_r$, the Baker--Campbell--Hausdorff formula for $u$ and $v$ converges absolutely, with the map $a:B_r\times B_r\to \mathfrak{g}$ being smooth with bounded derivatives:
\begin{align}\label{a-bound}
    a(u,v)&=O_G(|u||v|),\nonumber \\
 a(u,v)-a(u',v')&=O_G((|u|+|u'|)|v-v'|+(|v|+|v'|)|u-u'|), ~\forall u,v,u',v'\in B_r.
\end{align}

We denote the adjoint map $\operatorname{Ad}:G\to \operatorname{Aut}(\mathfrak{g})$, which is defined to be $\operatorname{Ad}_g= (dL_g)_{g^{-1}}\circ (dR_{g^{-1}})_e=(dR_{g^{-1}})_g\circ (dL_g)_e$ for $g\in G$. The derivative of the adjoint map, $\operatorname{ad}:\mathfrak{g}\to \operatorname{Der}(\mathfrak{g})$, is given by $\operatorname{ad}_v=[v,\cdot]$, for all $v\in \mathfrak{g}$. It is well-known that
\begin{equation}\label{adjexp}
\operatorname{Ad}_{\exp(v)}=e^{\operatorname{ad}_v},\quad \forall v\in \mathfrak{g},
\end{equation}
where the right-hand side denotes the element of $\operatorname{Aut}(\mathfrak{g})$ given by
\[
e^{\operatorname{ad}_v}\coloneqq \sum_{m=0}^\infty \frac{1}{m!}(\operatorname{ad}_v)^m.
\]

We will need to take derivatives under the exponential map. If $Y:I\to\mathfrak{g}$ is a differentiable curve in $\mathfrak{g}$ defined on some interval $I\subset \mathbb{R}$, then it is well-known that
\begin{equation}\label{diffexp}
    \frac{d}{dt}\exp(Y)=(dR_{\exp(Y)})_e\frac{e^{\operatorname{ad}_Y}-1}{\operatorname{ad}_Y}\left(\frac{dY}{dt}\right),
\end{equation}
where, on the right-hand side, the operator $\frac{e^{\operatorname{ad}_Y}-1}{\operatorname{ad}_Y}$ is defined using the power series expansion for $\frac{e^z-1}{z}$. By functional calculus, if $r>0$ is sufficiently small and $v\in B_r$, then the operator $\frac{e^{\operatorname{ad}_v}-1}{\operatorname{ad}_v}$ is invertible with the well-defined inverse $\frac{\operatorname{ad}_v}{e^{\operatorname{ad}_v}-1}$. Thus, if $Y(t)\in B_r$ for all $t\in I$, then 
denoting $X=\exp(Y)$, we have
\begin{equation}\label{difflog}
    \frac{d}{dt}\log(X)=\frac{\operatorname{ad}_Y}{e^{\operatorname{ad}_Y}-1}(dR_{X^{-1}})_X\left(\frac{dX}{dt}\right).
\end{equation}

Again by functional calculus, if $r>0$ is sufficiently small, we have
\begin{equation}\label{funccalc1}
    \frac{\operatorname{ad}_v}{e^{\operatorname{ad}_v}-1},\frac{\operatorname{ad}_v}{1-e^{-\operatorname{ad}_v}}=\operatorname{id}_\mathfrak{g}+O(|v|),\quad v\in B_r,
\end{equation}
and
\begin{equation}\label{funccalc2}
    \frac{\operatorname{ad}_v}{e^{\operatorname{ad}_v}-1}-\frac{\operatorname{ad}_w}{e^{\operatorname{ad}_w}-1},\frac{\operatorname{ad}_v}{1-e^{-\operatorname{ad}_v}}-\frac{\operatorname{ad}_w}{1-e^{-\operatorname{ad}_w}}=O(|v-w|),\quad v,w\in B_r,
\end{equation}
in the sense of operator norms. We will later use \eqref{funccalc1} and \eqref{funccalc2} to control \eqref{difflog}.

\section{Proof of Theorems \ref{idmainthm} and \ref{orbitalstabilitymainthm}}\label{sec:proof}
\setcounter{equation}{0}

In this section, we prove Theorems \ref{idmainthm} and \ref{orbitalstabilitymainthm}. Theorem \ref{idmainthm} is proven relatively directly assuming Theorem \ref{orbitalstabilitymainthm}.
\begin{proof}{\bf (Proof of Theorem \ref{idmainthm} assuming Theorem \ref{orbitalstabilitymainthm}.)}
Statement (1) of Theorem \ref{idmainthm} is a direct consequence of statements (2) and (3) of Theorem \ref{orbitalstabilitymainthm}.

In statement (2) of Theorem \ref{idmainthm}, statement (b) is a straightforward computation: defining $\tilde{X}_i(t)\coloneqq \exp(-Ht)X_i(t)$, we have $\tilde{X}_i(0)=X_i^0$ and
\[
dR_{\tilde{X}_i^{-1}}(\dot{\tilde{X}}_i(t))=\operatorname{Ad}_{\exp(-Ht)}\left(H+\frac{\kappa}{N}\sum_{j=1}^N \phi(X_j X_i^{-1})\right)-H=\frac{\kappa}{N}\sum_{j=1}^N \phi(\tilde{X}_j \tilde{X}_i^{-1}).
\]
Statements (a) and (c) are immediate from statements (2) and (3) of Theorem \ref{orbitalstabilitymainthm} for $\{\tilde{X}_i(t)\}_{i=1}^N$.
\end{proof}

Theorem \ref{orbitalstabilitymainthm} is proven by two reductions. First, we will simplify Theorem \ref{orbitalstabilitymainthm} into the more dynamically tractable statement of Theorem \ref{reduction}. Then, to prove Theorem \ref{reduction}, instead of working with the system \eqref{genLohe} on the Lie group $G$, we will work with a system, namely \eqref{logarithms} below, on the Lie algebra $\mathfrak{g}$ of the logarithms of the ratios of the particles. Proving Theorem \ref{reduction} on the Lie group $G$ will reduce to proving the corresponding Theorem \ref{Ymainthm} on the Lie algebra $\mathfrak{g}$. We will then prove Theorem \ref{Ymainthm} by introducing two Lyapunov functionals and establishing Gronwall inequalities for them.

\subsection{Reduction of Theorem \ref{orbitalstabilitymainthm}.}
We now begin the the proof of Theorem \ref{orbitalstabilitymainthm}. We first remark on the order of constants. We first select $C$ and $c_2$ depending on $G$, $\mathfrak{g}$, and $\phi$, then select $c_1$ sufficiently small depending on $C$ and $c_2$, then select $c$ sufficiently large depending on $C$, $c_2$, and $c_1$.

We first reduce Theorem \ref{orbitalstabilitymainthm} to the following theorem.
\begin{theorem}\label{reduction}
There exist a universal constant $c_3>0$, a sufficiently large constant $c>0$ and a sufficiently small constant $c_1>0$ depending on $G$, $\mathfrak{g}$, and $\phi$ such that if
\begin{equation*}
    \kappa>c \|H\|_\infty,
\end{equation*}
then the following holds.

\begin{enumerate}[(1)]
    \item (Local stability) Let $\{X_i(t)\}_{i=1}^N$ be a solution to \eqref{genLohe} with initial data $\{X_i^0\}_{i=1}^N$ satisfying
\[
    d(X_i^0,X_j^0)<c_1,\quad\forall i,j=1,\cdots,N.
\] Then there exists a unique global solution $X(t)=\{X_i(t)\}_{i=1}^N$ to \eqref{genLohe}, with
    \[
    d(X_i(t),X_j(t))<c_1,~t\ge 0,\quad \forall i,j=1,\cdots,N,
    \]
    and
    \[
    \limsup_{t\to\infty}d(X_i(t),X_j(t))\le \frac{c_3\|H\|_\infty}{\kappa\lambda},~t\ge 0,\quad \forall i,j=1,\cdots,N,
    \]
    where $\lambda$ is given as in \eqref{pos}.
    \item (Convergence of two flows) Let $\{\tilde{X_i}\}_{i=1}^N$ be a solution to \eqref{genLohe} with initial data $\{\tilde{X}_i^0\}_{i=1}^N$ satisfying
\begin{equation}\label{virtualSmall}
    d(\tilde{X}_i^0,\tilde{X}_j^0)<c_1,\quad\forall i,j=1,\cdots,N.
\end{equation}
    Then
    \[
        d(X_i(t)X_j(t)^{-1},\tilde{X}_i(t)\tilde{X}_j(t)^{-1})=O_G(d(X_i^0(X_j^0)^{-1},\tilde{X}_i^0(\tilde{X}_j^0)^{-1})e^{-\kappa\lambda t/3}).
    \]
\end{enumerate}
\end{theorem}
\begin{proof}{\bf (Proof of Theorem \ref{orbitalstabilitymainthm} assuming Theorem \ref{reduction}.)}
We prove Theorem \ref{orbitalstabilitymainthm} in the following order.
\begin{itemize}
    \item[(2)] We note that Theorem \ref{orbitalstabilitymainthm} (2) is just Theorem \ref{reduction} (1).
    \item[(1)+(3)] We prove statements (1) and (3) together.
    
    Let $\{\tilde{X}_i\}_{i=1}^N$ be a solution to \eqref{genLohe} with initial data $\{\tilde{X}_i^0\}_{i=1}^N$ satisfying
    \[
        d(\tilde{X}_i^0,\tilde{X}_j^0)<c_1,\quad\forall i,j=1,\cdots,N.
    \]
    By Theorem \ref{reduction} (2), we know that
    \begin{equation}\label{ratioexpsync}
        d(X_i(t)X_j(t)^{-1},\tilde{X}_i(t)\tilde{X}_j(t)^{-1})\le C e^{-\kappa \lambda t/3},
    \end{equation}
    for some constant $C$ depending on $G$. By considering the time-delayed solution $\tilde{X}(t)=X(t+1)$ to \eqref{genLohe}, Theorem \ref{reduction} (1) tells us that \eqref{virtualSmall} is true, so \eqref{ratioexpsync} tells us that $\{X_i(n)X_j(n)^{-1}\}_{n\in \mathbb{N}}$ is a Cauchy sequence and hence is exponentially convergent. By considering the time-delayed solutions $\tilde{X}(t)=X(t+T)$ for $0\le T\le 1$, again Theorem \ref{reduction} (1) tells us that \eqref{virtualSmall} is true, so we have by \eqref{ratioexpsync} that
\begin{align*}
&\sup_{0\le T\le 1}d(X_i(n)X_j(n)^{-1},X_i(n+T)X_j(n+T)^{-1})\\
&\le \sup_{0\le T\le 1} d(X_i(0)X_j(0)^{-1},X_i(T)X_j(T)^{-1})\cdot Ce^{-\kappa \lambda n/3}\\
&\le O_G(Ce^{-\kappa \lambda n/3}),\quad n\in \mathbb{N},
\end{align*}
which tells us that $\{X_i(t)X_j(t)^{-1}\}_{t\ge 0}$ is convergent to some value $X_{ij}^\infty$, $i,j=1,\cdots,N$.

Now it remains to verify the existence of a phase-locked state $X_i^\infty \exp (\Lambda t)$ satisfying (1), and that $X_{ij}=X_i^\infty (X_j^\infty)^{-1}$.

Note that, as $\frac{d}{dt}X^{-1}=-(dL_{X^{-1}})_e(dR_{X^{-1}})_X\dot{X}$,
\begin{equation*}
\frac{d}{dt}X_i^{-1}=-(dL_{X_i^{-1}})_e\left(H_i+\frac{\kappa}{N}\sum_{k=1}^N \phi(X_k X_i^{-1})\right).
\end{equation*}
So
\begin{align}\label{ratioODE}
\begin{aligned}
\frac{d}{dt}\left(X_i X_j^{-1}\right)=&(dR_{X_j^{-1}})_{X_i}\left(\frac{d}{dt}X_i\right)+(dL_{X_i})_{X_j^{-1}}\left(\frac{d}{dt}X_j^{-1}\right)\\
=&(dR_{X_iX_j^{-1}})_e\left(H_i+\frac{\kappa}{N}\sum_{k=1}^N \phi(X_k X_i^{-1})\right)\\
&-(dL_{X_iX_j^{-1}})_e\left(H_j+\frac{\kappa}{N}\sum_{k=1}^N \phi(X_k X_j^{-1})\right).
\end{aligned}
\end{align}

Since the ratios $X_iX_j^{-1}$ converge, both sides of \eqref{ratioODE} must converge to zero. Hence, for $i,j=1,\cdots,N$,
\begin{equation}\label{ratioasympequality}
(dR_{X_{ij}^{\infty}})_e\left(H_i+\frac{\kappa}{N}\sum_{k=1}^N \phi(X_{ki}^{\infty})\right)=(dL_{X_{ij}^{\infty}})_e\left(H_j+\frac{\kappa}{N}\sum_{k=1}^N \phi(X_{kj}^{\infty})\right).
\end{equation}

Now, by continuity, we must have $X_{ii}^\infty = e$ and $X_{ij}^\infty X_{jk}^\infty=X_{ik}^\infty$, $i,j,k=1,\cdots,N$. Therefore if we set
\[
X_i^\infty \coloneqq X_{i1}^\infty,\quad i=1,\cdots,N,
\]
then $X_{ij}=X_i^\infty(X_j^\infty)^{-1}$, and \eqref{ratioasympequality} becomes 
\begin{align*}
&(\operatorname{Ad}_{X_i^\infty})^{-1}\left(H_i+\frac{\kappa}{N}\sum_{k=1}^N \phi(X_{k}^{\infty}(X_i^\infty)^{-1})\right)\\
&=(\operatorname{Ad}_{X_j^\infty})^{-1}\left(H_j+\frac{\kappa}{N}\sum_{k=1}^N \phi(X_{k}^{\infty}(X_j^\infty)^{-1})\right),\quad i,j=1,\cdots,N.
\end{align*}
Thus we can define the common value $\Lambda\in \mathfrak{g}$ as
\[
\Lambda = (\operatorname{Ad}_{X_i^\infty})^{-1}\left(H_i+\frac{\kappa}{N}\sum_{k=1}^N \phi(X_{k}^{\infty}(X_i^\infty)^{-1})\right),\quad i=1,\cdots,N.
\]
Then $\Lambda$ and $X_i^\infty$ satisfy \eqref{pls}, so by Proposition \ref{Prop:pls}, $\{X_i^\infty \exp(\Lambda t)\}_{i=1}^N$ is indeed a phase-locked state.

We have just proved (3) and the existence statement of (1). It remains to verify the uniqueness statement of (1).

Let $\{\tilde{X}_i^\infty \exp(\tilde{\Lambda} t)\}_{i=1}^N$ be another phase-locked state, with
\[
d(\tilde{X}_i^\infty, \tilde{X}_j^\infty)<c_1.
\]
By Theorem \ref{reduction} (2), we have
\[
    X_i^\infty (X_j^\infty)^{-1}=\tilde{X}_i^\infty(\tilde{X}_j^\infty)^{-1},\quad i,j=1,\cdots,N.
\]
So $\tilde{X}_i^\infty(X_i^\infty)^{-1}=\tilde{X}_j^\infty(X_j^\infty)^{-1}$, $i,j=1,\cdots,N$, and if we set $g$ as this common value, we have 
\[
\tilde{X}_i^\infty=X_i^\infty g,\quad i=1,\cdots,N.
\]
Now from \eqref{pls}, we have $\operatorname{Ad}_{X_i^\infty}\Lambda = \operatorname{Ad}_{X_i^\infty g}\tilde{\Lambda}$, so
\[
    \tilde{\Lambda} = \operatorname{Ad}_{g^{-1}}\Lambda.
\]
Therefore
\[
\tilde{X}_i^\infty \exp(\tilde{\Lambda}t)=X_i^\infty g \exp(\operatorname{Ad}_{g^{-1}}\Lambda t)=X_i^\infty \exp(\Lambda t)g,\quad i=1,\cdots,N, ~t\ge 0,
\]
as desired.
    
\item[(4)] By (3), we may define the quantities
\[
\Lambda_i = H_i+\frac{\kappa}{N}\sum_{j=1}^N\phi\left(\lim_{t\to\infty}X_j(t)X_i(t)^{-1}\right),\quad i=1,\cdots,N.
\]
(By the discussion of the proof of (3), $\Lambda_i=\operatorname{Ad}_{X_i^\infty}\Lambda$.) 
By (2), we have $|\Lambda_i|\le O_{G,\phi}(\|H\|_\infty)$ for some constant $C$.

Now define 
\[
Z_i(t)=\exp(-\Lambda_it)X_i(t),\quad i=1,\cdots,N.
\]
Then by \eqref{adjexp},
\begin{align*}
(dR_{Z_i^{-1}})_{Z_i}\dot{Z}_i&=e^{-\mathrm{ad}_{\Lambda_i t}}\left(H_i+\frac{\kappa}{N}\sum_{j=1}^N\phi(X_jX_i^{-1})-\Lambda_i\right)\\
&=O_{G,\phi}(e^{O_G(\|H\|_\infty)t-\kappa\lambda t/3}),
\end{align*}
so if $\kappa>\Omega_G(\|H\|_\infty/\lambda)$, $Z_i$ exponentially converges, say to $Z_i^\infty$.

Now
\begin{align*}
d(X_i,\exp(\Lambda_i t)Z_i^\infty)=&d(\exp(\Lambda_i t)Z_i\exp(-\Lambda_i t),\exp(\Lambda_i t)Z_i^\infty\exp(-\Lambda_i t))\\
\le &e^{\|\operatorname{ad}_{\Lambda_i}\|t}d(Z_i,Z_i^\infty) \\
=&e^{O_G(\|H\|_\infty) t}O_{G,\phi}(e^{O_G(\|H\|_\infty)t-\kappa\lambda t/3})\\
\rightarrow & 0,
\end{align*}
where the first equality follows from right-invariance of $d$, and the first inequality follows from \eqref{adjexp} ($\|\operatorname{ad}_{\Lambda_i}\|$ is the operator norm of $\operatorname{ad}_{\Lambda_i}$). As $\exp(\Lambda_i t)Z_i^\infty=Z_i^\infty \exp(\operatorname{Ad}_{(Z_i^\infty)^{-1}}(\Lambda_i)t)$, we have
\begin{align*}
    &d(Z_i^\infty \exp(\operatorname{Ad}_{(Z_i^\infty)^{-1}}(\Lambda_i)t)(Z_j^\infty \exp(\operatorname{Ad}_{(Z_j^\infty)^{-1}}(\Lambda_j)t))^{-1}, \lim_{t\to\infty} X_i(t)X_j(t)^{-1})\\
    &\le d(\exp(\Lambda_i t)Z_i^\infty (\exp(\Lambda_j t)Z_j^\infty)^{-1},\exp(\Lambda_i t)Z_i^\infty X_j(t)^{-1})\\
    &\quad +d(\exp(\Lambda_i t)Z_i^\infty X_j(t)^{-1}, X_i(t)X_j(t)^{-1})\\
    &\quad +d(X_i(t)X_j(t)^{-1},\lim_{t\to\infty} X_i(t)X_j(t)^{-1})\\
    &\le e^{\|\operatorname{ad}_{\Lambda_i}\|t}\|\operatorname{Ad}_{Z_i^\infty}\|d((\exp(\Lambda_j t)Z_j^\infty)^{-1},X_j(t)^{-1})\\
    &\quad +d(\exp(\Lambda_i t)Z_i^\infty, X_i(t)) +d(X_i(t)X_j(t)^{-1},\lim_{t\to\infty} X_i(t)X_j(t)^{-1})\\
    &\le e^{\|\operatorname{ad}_{\Lambda_i}\|t}\|\operatorname{Ad}_{Z_i^\infty}\|e^{\|\operatorname{ad}_{\Lambda_j}\|t}\|\operatorname{Ad}_{Z_j^\infty}\|d(\exp(\Lambda_j t)Z_j^\infty,X_j(t))\\
    &\quad +d(\exp(\Lambda_i t)Z_i^\infty, X_i(t)) +d(X_i(t)X_j(t)^{-1},\lim_{t\to\infty} X_i(t)X_j(t)^{-1}).
\end{align*}
The last two terms clearly converge to zero, while the first term is $O_G(e^{O_G(\|H\|_\infty) t})O_{G,\phi}(e^{O_G(\|H\|_\infty)t-\kappa\lambda t/3})$ and thus converges to zero with large $\kappa$. Thus we have the convergence
\begin{align*}
Z_i^\infty \exp(\operatorname{Ad}_{(Z_i^\infty)^{-1}}(\Lambda_i)t)(Z_j^\infty \exp(\operatorname{Ad}_{(Z_j^\infty)^{-1}}(\Lambda_j)t))^{-1}\to \lim_{t\to\infty} X_i(t)X_j(t)^{-1},\\
 t\to\infty,
\end{align*}
for all $i,j=1\cdots,N$, which is only possible if $Z_i^\infty (Z_j^\infty)^{-1}=\lim_{t\to\infty}X_i(t)X_j(t)^{-1}$ and  $\operatorname{Ad}_{(Z_i^\infty)^{-1}}(\Lambda_i)=\operatorname{Ad}_{(Z_j^\infty)^{-1}}(\Lambda_j)$, $i,j=1,\cdots,N$. If we denote this latter common value by $M\in \mathfrak{g}$, then $\operatorname{Ad}_{Z_i^\infty}M=\Lambda_i=H_i+\frac{\kappa}{N}\sum_j \phi(Z_j^\infty (Z_i^\infty)^{-1})$ satisfies \eqref{pls}. Therefore $\{Z_i^\infty \exp(Mt)\}_{i=1}^N$ is a phase-locked state such that $d(X_i(t),Z_i^\infty \exp(Mt))$ converges exponentially to zero.

\item[(5)] We compute
\begin{align*}
    &(dL_{X_i(t)}^{-1})_e\dot{X}_i(t)-M\\
    &=\operatorname{Ad}_{X_i(t)^{-1}}\left(H_i+\frac{\kappa}{N}\sum_{j=1}^N\phi(X_jX_i^{-1})\right)-M\\
    &= (\operatorname{Ad}_{X_i(t)^{-1}Z_i^\infty\exp(Mt)}-I)\operatorname{Ad}_{(Z_i^\infty\exp(Mt))^{-1}}\left(H_i+\frac{\kappa}{N}\sum_{j=1}^N\phi(X_jX_i^{-1})\right)\\
    &\quad +\operatorname{Ad}_{(Z_i^\infty\exp(Mt))^{-1}}\left(H_i+\frac{\kappa}{N}\sum_{j=1}^N\phi(X_jX_i^{-1})-\Lambda_i\right)\\
    &\quad(\because \operatorname{Ad}_{(Z_i^\infty\exp(Mt))^{-1}}\Lambda_i=\operatorname{Ad}_{\exp(-Mt)}M=M).
\end{align*}
The first term is of the form $O_{G,\phi}(\|H\|e^{(O_G(\|H\|_\infty)-\kappa\lambda/3)t})$ and thus converges to zero. Similarly, the second term is of the form $O(e^{(O_G(\|H\|_\infty)-\kappa\lambda/3)t})$ and thus converges to zero.
\end{itemize}
\end{proof}

\subsection{Reducing from a model on the Lie group to a model on the Lie algebra.}
Our goal now is to prove Theorem \ref{reduction}. Theorem \ref{reduction} roughly says that if the particles $X_i$ are initially close together, they will stay together for all time and will exhibit good stability behavior. However, we have a priori little control on where the whole aggregate will head toward; it is possible that the entire population heads away from its initial position exponentially fast.

Thus, if we restrict our attention to the ratios $X_iX_j^{-1}$ of the particles, we should be able to analyze them in a small neighborhood of the identity.

From \eqref{ratioODE}, we may formulate the Cauchy problem for the relative phases $X_iX_j^{-1}$:
\begin{equation}\label{fractions}
\begin{cases}
\displaystyle
\frac{d}{dt}(X_i X_j^{-1})=(dR_{X_iX_j^{-1}})_e\Bigg(H_i-\operatorname{Ad}_{X_iX_j^{-1}}H_j\\
\qquad\qquad\qquad\qquad\qquad\qquad+\frac{\kappa}{N}\sum_{k=1}^N [\phi(X_k X_i^{-1})-\operatorname{Ad}_{X_iX_j^{-1}}\phi(X_k X_j^{-1})]\Bigg),\\
X_iX_j^{-1}(0)=X_i^0(X_j^0)^{-1},\quad i,j=1,\cdots,N.
\end{cases}
\end{equation}
Again, the standard Cauchy--Lipschitz theory guarantees the existence of a local $C^1$-solution to \eqref{fractions}. Observe that a global solution to \eqref{genLohe} exists if and only if a global solution to \eqref{fractions} exists. Indeed, the only if direction is trivial, and conversely if the $X_iX_j^{-1}$'s are globally defined, we may consider the right-hand side of \eqref{genLohe} to be a time-dependent bounded $C^1$-forcing term, and so by a standard argument (using the fact that $G$ as a manifold is complete) the solutions $X_i(t)$ must exist for all time.

Assuming $X_iX_j^{-1}\in \exp(B_r)$ for all $i,j=1,\cdots,N$, define $Y_{ij}=\log (X_iX_j^{-1})$. We will analyze the $Y_{ij}$'s on $\mathfrak{g}$ instead of the $X_iX_j^{-1}$'s on $G$; we are using $B_r$ as a coordinate chart for $\exp(B_r)$. Now, by \eqref{adjexp} and \eqref{difflog}, the $Y_{ij}$'s are the solution to the Cauchy problem
\begin{equation}\label{logarithms}
\begin{cases}
\displaystyle
\dot{Y}_{ij}=\frac{\operatorname{ad}_{Y_{ij}}}{e^{\operatorname{ad}_{Y_{ij}}}-1}H_i-\frac{\operatorname{ad}_{Y_{ij}}}{1-e^{-\operatorname{ad}_{Y_{ij}}}}H_j\\
\qquad\quad+\frac{\kappa}{N}\sum_{k=1}^N \left[\frac{\operatorname{ad}_{Y_{ij}}}{e^{\operatorname{ad}_{Y_{ij}}}-1}\phi\circ\exp(Y_{ki})-\frac{\operatorname{ad}_{Y_{ij}}}{1-e^{-\operatorname{ad}_{Y_{ij}}}}\phi\circ\exp(Y_{kj})\right],\\
Y_{ij}(0)=\log(X_i^0(X_j^0)^{-1}),\quad i,j=1,\cdots,N.
\end{cases}
\end{equation}
Once again, the standard Cauchy--Lipschitz theory guarantees the existence of a local $C^1$-solution to \eqref{logarithms}. This time, we can only observe that the existence of a global solution to \eqref{logarithms} guarantees the existence of a global solution to \eqref{fractions}, but not the other way around.

Note that, by the Baker-Campbell-Hausdorff formula, the $Y_{ij}$'s satisfy the compatibility equations
\begin{equation}\label{compatibility}
    Y_{ji}=-Y_{ij}, \quad Y_{ik}=Y_{ij}+Y_{jk}+a(Y_{ij},Y_{jk}),\quad i,j,k=1,\cdots,N
\end{equation}
for all times of their existence.

We now translate Theorem \ref{reduction}, a statement about the $X_i$'s, into a statement about the $Y_{ij}$'s.

\begin{theorem}\label{Ymainthm}
There exist a universal constant $C>0$, a sufficiently large constant $c>0$ and a sufficiently small constant $c_1>0$ depending on $G$, $\mathfrak{g}$, and $\phi$ such that if
\begin{equation}
    \kappa>c \|H\|_\infty
\end{equation}
then the following holds.
\begin{enumerate}[(1)]
    \item (Local stability)  If
    \[
    |Y_{ij}(0)|<c_1,\quad\forall i,j=1,\cdots,N,
    \]
    then there exists a unique global solution $Y(t)=\{Y_{ij}(t)\}_{i,j=1}^N$ to \eqref{logarithms} satisfying
    \[
    |Y_{ij}(t)|<c_1,~t\ge 0,\quad \forall i,j=1,\cdots,N,
    \]
    and
    \[
    \limsup_{t\to \infty}|Y_{ij}(t)|\le \frac{C\|H\|_\infty}{\kappa\lambda },\quad \forall i,j=1,\cdots,N.
    \]
    \item (Convergence of two flows) Let $\tilde{Y}(t)=\{\tilde{Y}_{ij}(t)\}_{i,j=1}^N$ be a solution to \eqref{logarithms} with initial data $\{\tilde{Y}_{ij}^0\}_{i,j=1}^N$ satisfying the compatibility equations \eqref{compatibility} and
    \[
        |\tilde{Y}_{ij}^0|<c_1,\quad \forall i,j=1,\cdots,N.
    \]
    Then
    \[
        |Y_{ij}(t)-\tilde{Y}_{ij}(t)|\le |Y_{ij}^0-\tilde{Y}_{ij}^0|\cdot e^{-\kappa \lambda t/3},\quad t\ge 0.
    \]
    
\end{enumerate}
\end{theorem}
It is clear that Theorem \ref{Ymainthm} implies Theorem \ref{reduction}, so our goal is now to prove Theorem \ref{Ymainthm}.

\subsection{Two Gronwall inequalities.}
The proof of Theorem \ref{Ymainthm} depends on two Gronwall inequalities for two Lyapunov functionals for $Y$, each respectively corresponding to statements (1) and (2) of Theorem \ref{Ymainthm}.

First, we consider the Lyapunov functional
\begin{equation}\label{Lyapunov1}
    \|Y\|_\infty=\max_{i,j=1,\cdots,N}|Y_{ij}|.
\end{equation}

\begin{lemma}[First Gronwall inequality]\label{FirstGronwall}
Let $Y(t)=\{Y_{ij}(t)\}_{i,j=1}^N$, $t\in [0,T)$, be a solution to \eqref{logarithms}, with $Y_{ij}(t)\in B_r$ for all $i,j=1,\cdots,N$ and $t\in [0,T)$. Then
\begin{equation}\label{firstGronwallprim}
    \dot{Y}_{ij}=-\kappa \left[d\phi_e(Y_{ij})+o_{G,\phi}(\|Y\|_\infty)\right]+O(\|H\|_\infty),\quad i,j=1,\cdots,N,
\end{equation}
and thus
\begin{equation}\label{firstGronwall}
    \left.\frac{d}{dt}\right|^+\|Y\|_\infty\le -\kappa \left[\lambda \|Y\|_\infty+o_{G,\phi}(\|Y\|_\infty)\right]+O(\|H\|_\infty),
\end{equation}
where $\left.\frac{d}{dt}\right|^+$ denotes the upper right Dini derivative, and the little $o$ notation is with respect to $\|Y\|_\infty\to 0$ and whose constants depend on $G$, $\mathfrak{g}$, and $\phi$.
\end{lemma}
\begin{remark}
Intuitively, since the original model \eqref{genLohe} is locally attractive, it is expected and unsurprising that a measure such as $\|Y\|_\infty$ for the diameter of a concentrated population should stay small; the `first-order force' $-\kappa d\phi_e(Y_{ij})$ in \eqref{firstGronwallprim} represents this local attraction.
\end{remark}
\begin{proof}{\bf (Proof of Lemma \ref{FirstGronwall}.)}
By functional calculus \eqref{funccalc1},
\begin{equation*}
\frac{\operatorname{ad}_{Y_{ij}}}{e^{\operatorname{ad}_{Y_{ij}}}-1}, \frac{\operatorname{ad}_{Y_{ij}}}{1-e^{-\operatorname{ad}_{Y_{ij}}}}=\operatorname{id}_\mathfrak{g}+O(|Y_{ij}|),
\end{equation*}
and
\[
\phi\circ\exp(Y_{ij}) = d\phi_e(Y_{ij})+o_{G,\phi}(|Y_{ij}|),
\]
so by \eqref{logarithms}
\[
\dot{Y}_{ij}=O(\|H\|_\infty)+\frac{\kappa}{N}\sum_{k=1}^N \left[d\phi_e(Y_{ki})-d\phi_e(Y_{kj})+o_{G,\phi}(\|Y\|_\infty)\right].
\]
Applying the Baker-Campbell-Hausdorff formula \eqref{compatibility},
\[
Y_{ki}-Y_{kj}\stackrel{\eqref{compatibility}}{=}-Y_{ij}-a(Y_{jk},Y_{ki})\stackrel{\eqref{a-bound}}{=}-Y_{ij}+O_G(\|Y\|_\infty^2),
\]
so
\[
\dot{Y}_{ij}=O(\|H\|_\infty)-\kappa\left[\phi_e(Y_{ij})+o_{G,\phi}(\|Y\|_\infty)\right]
\]
which is \eqref{firstGronwallprim}. Now \eqref{firstGronwall} follows from \eqref{firstGronwallprim} by \eqref{pos}.
\end{proof}

This Gronwall inequality allows us to prove the first part of Theorem \ref{Ymainthm}.
\begin{proof}{\bf (Proof of Theorem \ref{Ymainthm} (1).)}
Let $c_1$ be small enough so that $|Y_{ij}|<c_1$, $i,j=1,\cdots,N$, guarantees that $Y_{ij}\in B_r$ and that the $o(\|Y\|_\infty)$ in \eqref{firstGronwall} is smaller than $\frac{\lambda}{2}\|Y\|_\infty$. Then \eqref{firstGronwall} becomes
\[
\left.\frac{d}{dt}\right|^+\|Y(t)\|_\infty\le -\frac{\kappa \lambda}{2} \|Y(t)\|_\infty+C\|H\|_\infty
\]
for some absolute constant $C$, for all times $t\ge 0$ such that $\|Y(t)\|<c_1$. If $\kappa$ is large enough so that
\[
\frac{2C\|H\|_\infty}{\kappa\lambda}<c_1,
\]
then by a standard exit-time argument, $Y(t)<c_1$ for all $t\ge 0$ and $\limsup_{t\to \infty}\|Y(t)\|_\infty \le \frac{2C\|H\|_\infty}{\kappa\lambda}$.
\end{proof}

Our second Lyapunov functional measures the maximal mismatch between the two flows:
\begin{equation}\label{Lyapunov2}
    \|Y-\tilde{Y}\|_\infty=\max_{i,j=1,\cdots,N}|Y_{ij}-\tilde{Y}_{ij}|.
\end{equation}

\begin{lemma}[Second Gronwall inequality]\label{SecondGronwall}
Let $Y(t)=\{Y_{ij}(t)\}_{i,j=1}^N$ and $\tilde{Y}(t)=\{\tilde{Y}_{ij}(t)\}_{i,j=1}^N$, $t\in [0,T)$ be solutions to \eqref{logarithms} with initial data $\{Y_{ij}^0\}_{i,j=1}^N$ and $\{\tilde{Y}_{ij}^0\}_{i,j=1}^N$, respectively, with $Y_{ij}(t),\tilde{Y}_{ij}(t)\in B_r$ for all $i,j=1,\cdots,N$ and $t\in [0,T)$. Then
\begin{equation}\label{secondGronwallprim}
    \dot{Y}_{ij}-\dot{\tilde{Y}}_{ij}=-\kappa \left[d\phi_e(Y_{ij}-\tilde{Y}_{ij})+o_{G,\phi}(\|Y-\tilde{Y}\|_\infty)\right]+O(\|Y-\tilde{Y}\|_\infty\cdot \|H\|_\infty),
\end{equation}
and thus
\begin{equation}\label{secondGronwall}
    \left.\frac{d}{dt}\right|^+ \|Y-\tilde{Y}\|_\infty \le -\kappa \left[\lambda \|Y-\tilde{Y}\|_\infty+o_{G,\phi}(\|Y-\tilde{Y}\|_\infty)\right]+O(\|Y-\tilde{Y}\|_\infty\cdot \|H\|_\infty),
\end{equation}
where the little o notation is with respect to $\|Y\|,\|\tilde{Y}\|\to 0$.
\end{lemma}
\begin{remark}
Observe that \eqref{secondGronwallprim} is not the difference of \eqref{firstGronwallprim} for $Y$ and $\tilde{Y}$. One can accurately guess the leading term $-\kappa d\phi_e(Y_{ij}-\tilde{Y}_{ij})$ in this manner, but cannot get the correct error terms. It is crucial that the error terms are controlled in terms of $\|Y-\tilde{Y}\|$, as opposed to simply $\|Y\|$ and $\|\tilde{Y}\|$, in order to use \eqref{secondGronwall} to fruition; this fact is the core mathematical content of the proof of Theorem \ref{reduction} (2).
\end{remark}
\begin{proof}{\bf (Proof of Lemma \ref{secondGronwall}.)}
We expand
\begin{align*}
 \dot{Y}_{ij}-\dot{\tilde{Y}}_{ij}
    &=\frac{\operatorname{ad}_{Y_{ij}}}{e^{\operatorname{ad}_{Y_{ij}}}-1}H_i-\frac{\operatorname{ad}_{Y_{ij}}}{1-e^{-\operatorname{ad}_{Y_{ij}}}}H_j\\
&\quad+\frac{\kappa}{N}\sum_{k=1}^N \left[\frac{\operatorname{ad}_{Y_{ij}}}{e^{\operatorname{ad}_{Y_{ij}}}-1}\phi\circ\exp(Y_{ki})-\frac{\operatorname{ad}_{Y_{ij}}}{1-e^{-\operatorname{ad}_{Y_{ij}}}}\phi\circ\exp(Y_{kj})\right]\\
    &\quad -\frac{\operatorname{ad}_{\tilde{Y}_{ij}}}{e^{\operatorname{ad}_{\tilde{Y}_{ij}}}-1}H_i+\frac{\operatorname{ad}_{\tilde{Y}_{ij}}}{1-e^{-\operatorname{ad}_{\tilde{Y}_{ij}}}}H_j\\
&\quad-\frac{\kappa}{N}\sum_{k=1}^N \left[\frac{\operatorname{ad}_{\tilde{Y}_{ij}}}{e^{\operatorname{ad}_{\tilde{Y}_{ij}}}-1}\phi\circ\exp(\tilde{Y}_{ki})-\frac{\operatorname{ad}_{\tilde{Y}_{ij}}}{1-e^{-\operatorname{ad}_{\tilde{Y}_{ij}}}}\phi\circ\exp(\tilde{Y}_{kj})\right]\\
    &= \left(\frac{\operatorname{ad}_{Y_{ij}}}{e^{\operatorname{ad}_{Y_{ij}}}-1}-\frac{\operatorname{ad}_{\tilde{Y}_{ij}}}{e^{\operatorname{ad}_{\tilde{Y}_{ij}}}-1}\right)\left(H_i+\frac{\kappa}{N}\sum_{k=1}^N \phi\circ\exp(\tilde{Y}_{ki})\right)\\
    &\quad -\left(\frac{\operatorname{ad}_{Y_{ij}}}{1-e^{-\operatorname{ad}_{Y_{ij}}}}-\frac{\operatorname{ad}_{\tilde{Y}_{ij}}}{1-e^{-\operatorname{ad}_{\tilde{Y}_{ij}}}}\right)\left(H_j+\frac{\kappa}{N}\sum_{k=1}^N \phi\circ\exp(\tilde{Y}_{kj})\right)\\
    &\quad +\frac{\kappa}{N}\sum_{k=1}^N \Bigg[\frac{\operatorname{ad}_{Y_{ij}}}{e^{\operatorname{ad}_{Y_{ij}}}-1}\Big(\phi\circ\exp(Y_{ki})-\phi\circ\exp(\tilde{Y}_{ki})\Big)\\
&\qquad\qquad\qquad-\frac{\operatorname{ad}_{Y_{ij}}}{1-e^{-\operatorname{ad}_{Y_{ij}}}}\Big(\phi\circ\exp(Y_{kj})-\phi\circ\exp(\tilde{Y}_{kj})\Big)\Bigg].
\end{align*}
By functional calculus \eqref{funccalc2},
\[
    \frac{\operatorname{ad}_{Y_{ij}}}{e^{\operatorname{ad}_{Y_{ij}}}-1}-\frac{\operatorname{ad}_{\tilde{Y}_{ij}}}{e^{\operatorname{ad}_{\tilde{Y}_{ij}}}-1}, \frac{\operatorname{ad}_{Y_{ij}}}{1-e^{-\operatorname{ad}_{Y_{ij}}}}-\frac{\operatorname{ad}_{\tilde{Y}_{ij}}}{1-e^{-\operatorname{ad}_{\tilde{Y}_{ij}}}}= O(|Y_{ij}-\tilde{Y}_{ij}|),
\]
and
\begin{align*}
    \phi\circ\exp(Y_{ij})-\phi\circ\exp(\tilde{Y}_{ij})=&(d\phi_{e}\circ d\exp_0)(Y_{ij}-\tilde{Y}_{ij})+o_{G,\phi}(|Y_{ij}-\tilde{Y}_{ij}|)\\
    =&d\phi_{e}(Y_{ij}-\tilde{Y}_{ij})+o_{G,\phi}(|Y_{ij}-\tilde{Y}_{ij}|).
\end{align*}
So we may simplify
\begin{align*}
    \dot{Y}_{ij}-\dot{\tilde{Y}}_{ij}    &=O(\|Y-\tilde{Y}\|_\infty(\|H\|_\infty+\kappa\|Y\|_\infty+\kappa\|\tilde{Y}\|_\infty))\\
    &\quad +\frac{\kappa}{N}\sum_{k=1}^N \left[\Big(\phi\circ\exp(Y_{ki})-\phi\circ\exp(\tilde{Y}_{ki})\Big)-\Big(\phi\circ\exp(Y_{kj})-\phi\circ\exp(\tilde{Y}_{kj})\Big)\right]\\
    &=O(\|Y-\tilde{Y}\|_\infty\|H\|_\infty)+\kappa\cdot o_G(\|Y-\tilde{Y}\|_\infty)\\
    &\quad +\frac{\kappa}{N}\sum_{k=1}^N \left[d\phi_e(Y_{ki}-\tilde{Y}_{ki}-Y_{kj}+\tilde{Y}_{kj})+o_{G,\phi}(\|Y-\tilde{Y}\|_\infty)\right].
\end{align*}
Again, by the Baker-Campbell-Hausdorff formula \eqref{compatibility},
\begin{align*}
    Y_{ki}-Y_{kj}-\tilde{Y}_{ki}+\tilde{Y}_{kj}&\stackrel{\eqref{compatibility}}{=}-Y_{ij}+ \tilde{Y}_{ij}-a(Y_{jk},Y_{ki})+a(\tilde{Y}_{jk},\tilde{Y}_{ki})\\
    &\stackrel{\eqref{a-bound}}{=}-Y_{ij}+ \tilde{Y}_{ij}+O_G((\|Y\|_\infty+\|\tilde{Y}\|_\infty)\|Y-\tilde{Y}\|_\infty),
\end{align*}
and this last big O term is again $o_G(\|Y-\tilde{Y}\|_\infty)$. Thus
\[
\dot{Y}_{ij}-\dot{\tilde{Y}}_{ij}=-\kappa \left[d\phi_e(Y_{ij}-\tilde{Y}_{ij})+o_{G,\phi}(\|Y-\tilde{Y}\|_\infty)\right]+O(\|Y-\tilde{Y}\|_\infty\cdot \|H\|_\infty).
\]
\end{proof}

\begin{proof}{\bf (Proof of Theorem \ref{Ymainthm}(2).)} If $c_1$ is small enough so that $\|Y\|,\|\tilde{Y}\|<c_1$ implies that $Y,\tilde{Y}\in B_r$ and that the $o(\|Y-\tilde{Y}\|_\infty)$ term in \eqref{secondGronwall} is at most $\frac{\lambda \|Y-\tilde{Y}\|_\infty}{3}$, \eqref{secondGronwall} becomes 
\[
    \left.\frac{d}{dt}\right|^+ \|Y-\tilde{Y}\|_\infty \le -\frac{2\kappa \lambda}{3} \|Y-\tilde{Y}\|_\infty+C\|Y-\tilde{Y}\|_\infty\cdot \|H\|_\infty
\]
for some absolute constant $C>0$. By choosing $\kappa>\frac{3C}{\lambda}\|H\|_\infty$ we have
\[
    \left.\frac{d}{dt}\right|^+ \|Y-\tilde{Y}\|_\infty \le -\frac{\kappa \lambda}{3} \|Y-\tilde{Y}\|_\infty
\]
for times $t$ at which $\|Y(t)\|,\|\tilde{Y}(t)\|<c_1$. This hypothesis is satisfied by Theorem \ref{Ymainthm} (1) if $\|Y(0)\|,\|\tilde{Y}(0)\|<c_1$, in which case we have
\begin{equation}\label{expsync}
\|Y(t)-\tilde{Y}(t)\|_\infty\le \|Y(0)-\tilde{Y}(0)\|_\infty \exp\left(-\frac{\kappa \lambda}{3}t\right).
\end{equation}

\end{proof}

\begin{remark}\label{chaosremark}
Question \ref{chaos} may be reformulated as follows. Is it possible that in a semi-local regime with intermediate coupling strength $\kappa$:
\[
c_1'<\max_{i,j=1,\cdots,N}d(X_i^0,X_j^0)<c_1,\quad c\|H\|_\infty <\kappa<c'\|H\|_\infty,
\]
with different constants $c_1'$ and $c'$, \eqref{firstGronwallprim} holds yet \eqref{secondGronwallprim} fails?
\end{remark}

\par{\bf Acknowledgements.}\, This work was written while the author was visiting Seoul National University under the auspices of the Hyperbolic and Kinetic Equations (HYKE) group. The author thanks Hangjun Cho and Myeongju Kang of the HYKE group for helpful discussions. This work was partially supported by the Korea Foundation for Advanced Studies. The author thanks the anonymous referee for valuable comments and suggestions.

\bibliographystyle{myabbrv}
\bibliography{main}

          \end{document}